\documentclass[11pt]{article}
\usepackage{amsthm}

\usepackage{graphicx}
\usepackage{mathbbol} %
\usepackage[english]{babel}
\usepackage{amssymb}
\usepackage{amsmath}
\usepackage{mathtools}
\usepackage{enumerate} 
\usepackage{url}
\usepackage[usenames,dvipsnames,svgnames,table]{xcolor}
\usepackage{tikz}%
\usepackage[normalem]{ulem}

\DeclareMathOperator    \argmin         {arg\,min}
\DeclareMathOperator    \argmax         {arg\,max}

\DeclareMathOperator    \iconv          {ic} %
\DeclareMathOperator    \conv           {conv}

\DeclareMathOperator    \dist           {dist}

\DeclareMathOperator    \intr                   {int}

\DeclareMathOperator    \rec                    {rec}

\DeclareMathOperator    \verts          {vert}
\DeclareMathOperator    \vol                    {vol}

\newcommand{\N}{\mathbb{N}}

\newcommand{\Q}{\mathbb{Q}}
\newcommand{\R}{\mathbb{R}}

\newcommand{\Z}{\mathbb{Z}}

\newcommand\st{:}

\newtheorem{theorem}{Theorem}[section]
\newtheorem{corollary}[theorem]{Corollary}
\newtheorem{lemma}[theorem]{Lemma}
\newtheorem*{lemma*}{Lemma}
\newtheorem{proposition}[theorem]{Proposition}
  \theoremstyle{remark}

  \theoremstyle{remark}

  \theoremstyle{remark}
\newtheorem{claim}[theorem]{Claim}

\newtheorem{definition}[theorem]{Definition}

\newcommand{\muconvex}{\mu_{\mathrm{c}}}

\renewcommand{\varepsilon}{\epsilon}

\renewcommand{\alpha}{g}

\begin{document}

\title{Sublinear Bounds for a Quantitative Doignon-Bell-Scarf Theorem}

\author{
Stephen R.\ Chestnut\thanks{Institute for Operations Research, Department of Mathematics, ETH Z\"urich, \texttt{stephenc@ethz.ch}}
\and
Robert Hildebrand\thanks{Mathematical Sciences Department, IBM T.\ J.\ Watson Research Center, \texttt{rhildeb@us.ibm.com}}\and
\and
Rico Zenklusen\thanks{Institute for Operations Research, Department of Mathematics, ETH Z\"urich, \texttt{ricoz@math.ethz.ch}. Supported by Swiss National Science Foundation grant 200021\_165866.}
}

\date{\today}

\maketitle
 
\begin{abstract}
The recent paper ``A quantitative Doignon-Bell-Scarf Theorem" by Aliev et al.\ generalizes the famous Doignon-Bell-Scarf Theorem on the existence of integer solutions to systems of linear inequalities.  Their generalization examines the number of facets of a polyhedron that contains exactly $k$ integer points in $\R^n$.  They show that there exists a number $c(n,k)$ such that any polyhedron in $\R^n$ that contains exactly $k$ integer points has a relaxation to at most $c(n,k)$ of its inequalities that will define a new polyhedron with the same integer points.  They prove that $c(n,k) = O(k)2^n$.  In this paper, we improve the bound asymptotically to be sublinear in $k$, that is, $c(n,k) = o(k) 2^n$.  We also provide lower bounds on $c(n,k)$, along with other structural results.  For dimension $n=2$, our upper and lower bounds match to within a constant factor.
\end{abstract}

\section{Introduction}
The classical theorem of Helly states that for any finite collection of convex subsets of $S = \R^n$, if the intersection of every $n+1$ subsets is nonempty, then the intersection of the entire collection is nonempty.  
In 1973, Doignon was curious as to how such a theorem could hold over the discrete set of integers $S = \Z^n$.  
What resulted is a famous theorem that can be phrased as follows:
any system of linear inequalities in $\R^n$ without integer solutions has a subsystem of at most $2^n$ inequalities that also has no integer solutions~\cite[Proposition 4.2]{doignon1973}.  
This result was also independently rediscovered shortly thereafter by both Bell~\cite{bell1977} and Scarf~\cite[Theorem 1.4]{scarf1977}.  %
  In its Helly formulation, the Doignon-Bell-Scarf Theorem states that for any finite collection of convex subsets of $\R^n$, if the intersection of every $2^{n}$ subsets contains at least one integer point, then the intersection of the entire collection contains at least one integer point.

Since Doignon's result, many versions of Helly's theorem have been studied based on the underlying set of feasible points.  For instance, the Helly number with $S = \Z^n \times \R^d$ is $(d+1)2^n$, as was known in~\cite{hoffman79} and was rediscovered in~\cite[Theorem 1.1]{averkov2012}.   See~\cite{amenta2015helly} for a recent survey of variations and applications of Helly's theorem.  
We focus on a quantitative generalization of Doignon's result guided by the following definition. 
\begin{definition}
Given integers $n\geq 1$ and $k\geq 0$, let $c(n,k)$ denote the least integer such that for any $m$, matrix $A\in\R^{m\times n}$, and  $b\in\R^m$, if the polyhedron $\{x\in\R^n \st Ax \leq b\}$ has exactly $k$ integer points, then there exists a subset $I$ of the rows of $A$ with $|I|\leq c(n,k)$ such that $\{x\in\R^n \st A_I x \leq b_I\}$ has exactly the same $k$ integer points, where $A_I$ is the submatrix of $A$ only containing the rows $I$.
\end{definition}
Phrased as a Helly-type theorem, this implies that if a finite collection of convex subsets of $\R^n$ has the property that the intersection of any $c(n,k)$ subsets contains at least $k+1$ integer points, then the intersection of all of them contains at least $k+1$ integer points~\cite{Aliev2016}.
Following this notation, the Doignon-Bell-Scarf Theorem asserts that, for all $n\geq 1$, $c(n,0) = 2^n$.  
The quantity $c(n,k)$ was formally defined by Aliev, De Loera, and Louveaux~\cite{alievIPCO2014}.  A much older result of Bell~\cite{bell1977} is equivalent to the inequality $c(n,k) \leq (k+2)^n$. 
Aliev et al.~\cite[Proof of Theorem 1]{alievIPCO2014} improve the upper bound to $c(n,k)\leq 2^k2^n$ and have since, together with Bassett~\cite[Theorem 1]{Aliev2016}, improved the bound to
\begin{equation}\label{eq: qdbs bound}
c(n,k) \leq \lceil 2(k+1)/3\rceil 2^n - 2 \lceil 2(k+1)/3\rceil + 2.
\end{equation}
This linear bound was recently improved by a constant factor in~\cite[Theorem 3]{averkov2016} to 
\begin{equation}
\label{eq:qdbs bound 2}
c(n, k) \leq \lfloor (k + 1)/2 \rfloor (2^n - 2) + 2^n.
\end{equation}

Our Theorem~\ref{thm: sublinear upper bound} sharpens the upper bound to $o(k)\cdot2^n$ and simplifies the reasoning behind the bound.
To illustrate out technique, let us quickly describe how one can achieve an upper bound of $(k+1)2^n$, which is of the same order as~\eqref{eq: qdbs bound} and~\eqref{eq:qdbs bound 2}.
It begins with a trick that was known already by Bell~\cite{bell1977}, and which is formalized in Lemma~\ref{lem: one point per facet}. 
Recall that a set of points is in convex position if it contains no point that can be expressed as a convex combination of the others.  For a set $V \subseteq \R^n$, denote $\iconv(V) \coloneqq \conv(V)\cap\Z^n$ where $\conv(V)$ is the convex hull of $V$.
Let  $P = \{x\in\R^n\st Ax\leq b\}$ and let $X = P \cap \Z^n$.   If $P$ contains exactly $k$ integer points and is defined by $c(n,k)$ constraints, none of which can be removed without affecting the integer hull, then there exists a set $V\subseteq\Z^n$ of $c(n,k)$ points in convex position such that $\iconv(V)\setminus V \subseteq P$.  That is, the non-vertex lattice points in $\conv(V)$ are contained in $P$.

To prove the upper bound, take $P$ and $V$ as above, and consider the parities of the points in $V$.  
Let $A\subseteq V$ be the set of points with a given parity, that is, with exactly the same coordinates even and odd. 
The midpoint of any pair of points in $A$ is an integer point in $P$, so the number of integer points $k$ in $P$ is lower bounded by the number of midpoints of $A$.  The number of midpoints is $|A+A|-|A|$, which can be lower bounded as $|A+A| - |A| \geq |A| - 1$.  Here,  $A+A$ is the Minkowski sum of $A$ with itself.  The inequality holds, for instance, because $A \subseteq \Z^n$.  Thus each parity class $A$ must have at most $k+1$ points in it.  As there are $2^n$ parity classes, we see that  $c(n,k)=|V|\leq (k+1)2^n$.

For dimension $n=2$, Erd\H{o}s, Fishburn, and F\"uredi~\cite[Theorem 1]{erdos91} proved an inequality, which relaxed,   implies that $|A+A|-|A|\geq \frac{1}{4}|A|^2$ for each finite subset $A\subseteq\R^2$ in convex position.
Applying the reasoning above to this case yields $c(2,k)\leq 8\sqrt{k}$.
The asymptotic estimate $c(2,k) = O(\sqrt{k})$ is not optimal.
Bell~\cite{bell1977} mentions that, by using a result of Andrews~\cite{andrews1963}, any polygon with at least $c$ facets, the removal of any one of which affects the integer hull, has $\Omega(c^3)$ interior lattice points.  
This argument implies that there exists a constant $C$ such that 
\begin{equation}
\label{eq:bell}
c(2,k) \leq C k^{\frac{1}{3}}.
\end{equation}
Theorem~\ref{thm:c(2 k)} gives one explicit choice of constant.

The Doignon-Bell-Scarf Theorem has had many applications in the theory of integer programming and the geometry of numbers.   
Aliev et al.~\cite[Corollary 3]{Aliev2016}\ use $c(n,k)$ to find the $k$ best solutions to an integer linear program by extending the randomized algorithm of Clarkson~\cite{clarkson1995}. 
Cutting planes have recently been studied through the lens of intersection cuts that are derived from maximal lattice-free polyhedra.
See, for instance,~\cite{corner_survey}.  
It follows from the Doignon-Bell-Scarf Theorem that any maximal lattice free polyhedron has at most $2^n$ facets. 
This result has been central in classifying the different types of maximal lattice free convex sets.  
In a similar way, $c(n,k)$ bounds the facet complexity of a maximal polyhedron containing $k$ integer points.

Baes, Oertel, and Weismantel used the Doignon-Bell-Scarf theorem as a basis for describing a dual for mixed-integer convex minimization~\cite{oertel2015}.  
Likewise, $c(n,k)$ can be used to describe a dual for the problem of finding the $k$~best solutions to an integer convex minimization problem. 

Finally, the techniques used for studying $c(n,k)$ in~\cite{Aliev2016} and in this paper involve analyzing lattice polytopes.  A lattice polytope is a polytope whose vertices are contained in $\Z^n$.  In particular, $c(n,k)$ is closely related to the maximum number of vertices of a lattice polytope that contains $k$ non-vertex lattice points.  %

\subsection*{Our Contributions}
This paper improves the asymptotic upper bound on $c(n,k)$ to sublinear in $k$ using the midpoints methodology outlined above and stronger bounds on the cardinality of \emph{sumsets}, the Minkowski sum of two finite sets, from the field of additive combinatorics.   This approach yields the bound as $c(n,k) = \frac{k}{e^{\Omega(\log^{1/7} k)}} 2^n = o(k)2^n$.

We complement the sublinear (in $k$) upper bound on $c(n,k)$ with a $\Omega(k^{\frac{n-1}{n+1}})$ lower bound, where the hidden constant depends on the dimension $n$.
The lower bound is proved by relating $c(n,k)$ to the maximum vertex complexity of a lattice polytope containing $k$ lattice points and then applying a theorem of B\'ar\'any and Larman~\cite{Barany1998} implying that the integer hull of the $n$-dimensional ball with radius $r$ has $\Theta(r^{n(n-1)/(n+1)})$ many vertices, where the neglected constants depend on $n$.

Increasing the lower bounds on $|A+A|$ might be one way to decrease our upper bound.
In dimension greater than two, surprisingly little is known about the minimum possible cardinality of $|A+A|$ for $A$ in convex position. In Proposition~\ref{prop:muconvex-bound}, we derive an upper bound for the minimum number of $|A + A|$ from the integer points in the $n$-dimensional sphere.
When proving the sublinear upper bound, in Theorem~\ref{thm: sublinear upper bound}, we use the convexity of the points only to say that the average of any five points is not contained in the set.  We are unaware
 of lower bounds on sumset cardinalities that completely make use of convexity of the points.  Progress on this front would immediately lead to an improved sublinear upper bound.

As the paper progresses towards proving the bounds, we establish some structural properties of $c(n,k)$ and related sequences that may be of independent interest.
One of the difficulties, in particular, that we encounter while proving the lower bound is that $c(n,k)$ is not monotonic in $k$, which we demonstrate by proving that $c(2,5)=7<8=c(2,4)$.

The formula given by Aliev et al.~\cite{Aliev2016} implies $c(n,2)\leq 2(2^n-1)$ and they left as an open problem whether equality holds.
Gonzalez Merino and Henze (see~\cite{amenta2015helly}) proved equality holds, and we do as well (Proposition~\ref{prop: c(n,2) lower bound}).  We also give an alternative proof that $c(n,2)\leq 2(2^n-1)$ (Theorem~\ref{thm: subspaces}).
Corollary~\ref{cor: large n} and Proposition~\ref{prop: c(n,2) lower bound}, with the discussion immediately following it, show that this is the correct asymptotic behavior in $n$ for every (fixed) $k$, more precisely $|c(n,k)-2^{n+1}|= O(1)$, for every $k\geq 1$.

After this paper was made available online, Averkov, Merino, Henze, Paschke, and Weltge~\cite{averkov2016} posted similar results.  Our definition of $c(n,k)$ corresponds to their definition of $c(\Z^n,k)$.  We will comment on similarities where they apply in this paper.  In particular, they prove the bound $c(n,k) \leq k^{\frac{n-1}{n+1}} \cdot (3n)^{5n}$~\cite[Theorem 4]{averkov2016}.  That bound has optimal dependence on $k$, but the dependence in $n$ is not optimal, so their bound is incomparable to the bound $c(n,k) \leq \frac{k}{e^{\Omega(\log^{1/7}k)}}\cdot 2^n$ proved here.  

Section~\ref{sec: bell} describes Bell's technique in more detail.
The midpoints technique is applied in Section~\ref{sec: upper bounds} to prove the sublinear upper bound and in Section~\ref{sec: lower bounds} we prove the lower bound.
Finally, Section~\ref{sec: nonmonotonic} proves that $c(n,k)$ is not a monotonic function of $k$ by proving exact values for $c(2,k)$ for certain choices of $k$.  Note that this can also be found in~\cite{averkov2016}, where they computationally find values of $c(2,k)$ for $k \in \{0,\dots, 30\}$. 

\section{Bell's Expanding Polyhedron}\label{sec: bell}
 
The main lemma of this section, Lemma~\ref{lem: one point per facet}, is an important tool for understanding the behavior of $c(n,k)$. 
It relates a general system of linear inequalities with a collection of integer points in convex position. 
Each integer point serves as a witness to one inequality in a system where every inequality is necessary to maintain the integer hull.  This is described in Lemma~\ref{lem: one point per facet}.  This idea is similar to the fact that all maximal lattice free convex sets are polyhedra and have at least one lattice point on the relative interior of each facet~\cite{Lovasz}.  See also~\cite[Theorem 16]{averkov2016} which shows that for all $k \geq 1$, all maximal convex sets with $k$ integer points in their interior are polytopes with at least one integer point in the relative interior of each facet.  
 
The basic idea of the proof we give is described by~\cite{bell1977}, and it uses several ideas from~\cite{Aliev2016}.  Similar techniques for related results were used in several other places including~\cite[Lemma 4.4]{averkov2013}.

We begin with a lemma about polyhedral cones and also a lemma about $\Q$-linear independence.  
We say a matrix $A \in \R^{n\times m}$ is \emph{generic} if all subsets of at most $n$ rows are linearly independent.  Let $\intr(K)$ denote the interior of a set $K$.
\begin{lemma}
\label{lem:cone-non-empty}
Let $A \in \R^{n \times m}$ be generic and let $C = \{x \st Ax \leq 0\}$.  If there exists a non-zero $v \in C$, then $\intr(C) \neq \emptyset$.
\end{lemma}
\begin{proof}
Let $a_1, \dots, a_m \in \R^n$ be the rows of $A$.
Let $I,J \subseteq [m]$ such that $a_i \cdot  v = 0$ for all $i \in I$ and $a_j\cdot v < 0$ for all $j \in J$.  Since $v$ is non-zero and $A$ is generic, $|I|<n$.  Let $C_I = \{x \st a_i \cdot x \leq 0, i \in I\}$, $C_J  = \{x \st a_j \cdot x < 0, j \in J\}$.  Since $v \in C_I \cap C_J \subseteq C$, and $v \in \intr(C_J)$, it is sufficient to show that $\intr(C_I) \neq \emptyset$.

If $\intr(C_I) = \emptyset$, then there exists a hyperplane $H = \{x \st h \cdot x = 0\}$, with $h$ non-zero, such that $C \subseteq H$.  This implies that $h = \sum_{i \in I} \lambda_i a_i$ and $-h = \sum_{i \in I} \mu_i a_i$ for some $\lambda_i, \mu_i \geq 0$.  But then $0 = h - h = \sum_{i \in I} (\lambda_i + \mu_i) a_i$, which is a contradiction since the vectors $a_i$ are linearly independent and $\lambda_i + \mu_i > 0$ for at least one $i \in I$.
\end{proof}

\begin{lemma}
\label{lem:q-linear-independence}
There exists a dense subset $U \subseteq \R^n$ such that any finite subset $\{u_1, \dots, u_k\}$ is linearly independent over $\Q$.  That is, there do not exist non-trivial vectors $q_1, \dots, q_k \in \Q^n$ such that $\sum_{i=1}^k q_i \cdot u_i = 0$.
\end{lemma}
\begin{proof}
We construct $U$ by starting with $\mathbb{Q}^n$, which is a countable set that is dense within $\mathbb{R}^n$, and then perturbing the vectors in $\mathbb{Q}^n$ to obtain the desired linear independence over $\mathbb{Q}$ without destroying the property of being dense in $\mathbb{R}^n$.

Since $\mathbb{Q}^n$ is countable, there is a numbering of its elements $\mathbb{Q}^n = \{q_i \mid i\in \mathbb{Z}_{\geq 1}\}$.
To perturb the elements $q_i$, let $\mathcal F$ be a basis for $\R$ as a vector space over $\Q$. Since $\R$ is uncountable and $\Q$ is countable, $|\mathcal F| = \infty$. Moreover, we choose a basis $\mathcal{F}$ such that $|f|\leq 1$ for $f\in \mathcal{F}$; indeed by starting with any basis $\mathcal{F}'$ and replacing each non-zero element $f'\in \mathcal{F}'$ by $f'/\lceil|f'|\rceil$, such a basis is obtained.
Let $\{f_1,f_2,\ldots \}\subseteq \mathcal{F}$ be any countably infinite subset of $\mathcal{F}$. We now group the elements of this subsequence into vectors of $\mathbb{R}^n$, by defining for $i\in \mathbb{Z}_{\geq 1}$, $f^i \coloneqq (f_{1+(i-1)n},\ldots, f_{in})$.
We now set
\begin{equation*}
U = \left\{q_i + \frac{1}{i} \cdot f^i \;\middle\vert\; i\in \mathbb{Z}_{\geq 1}\right\}\enspace.
\end{equation*}
By our choice of $f^i$, $U$ clearly has the desired linear independence properties over $\mathbb{Q}$. Hence, it remains to show that $U$ is dense in $\mathbb{R}^n$. To this end, let $r\in \mathbb{R}^n$. Because $\mathbb{Q}^n$ is dense in $\mathbb{R}^n$, there is an infinite subsequence of $q_1, q_2 ,\ldots$ that converges to $r$, i.e., there are indices $1 \leq i_1 < i_2 < \ldots$ such that $\lim_{j\rightarrow \infty} q_{i_j} = r$. For the subsequence of $U$ with the same indices we have
\begin{equation*}
\lim_{j\rightarrow \infty} \left(q_{i_j} + \frac{1}{i_j}\cdot f^{i_j}\right) = \lim_{j\rightarrow\infty} q_{i_j} = r\enspace,
\end{equation*}
where we use the fact that each component of $f^i$, for any $i\in \mathbb{Z}_{\geq 1}$, is at most $1$ in absolute value. This shows that $U$ is dense in $\mathbb{R}^n$ and thus completes the proof.
\end{proof}

For a set $X \subseteq \Z^n$, we say a system of inequalities $Ax\leq b$ is \emph{non-redundant} with respect to $X$ if 
 removing any one inequality strictly increases the number of solutions of $X$ that exist.  
We say that a polyhedron $P$ is a \emph{maximum non-redundant polyhedron} with respect to $X$ if there is no other non-redundant polyhedron $P'$ with more facets such that $P' \cap X = P \cap X$.  We will simply use the terms non-redundant and maximum non-redundant when $X = \Z^n$.

\begin{lemma}\label{lem: one point per facet}
Let $P = \{x \st a_i \cdot x \leq b_i, i=1, \dots, m\}$ be a non-empty maximum non-redundant polyhedron with $a_1, \dots, a_m \in \R^{n}$, $b \in \R^m$ such that $X = P \cap \Z^n$ is finite.  Then there exist $a'_1, \dots, a'_m \in \R^{n}, b' \in \R^m$ such that 
\begin{enumerate}
\item $P' = \{x \st a'_i \cdot x \leq b'_i, i=1, \dots, m\}$ is bounded with $\intr(P') \cap \Z^n = P \cap \Z^n$ and $P'$ has exactly one integer point on the relative interior of each of its $m$ facets.
\item $V = (P' \setminus P) \cap \Z^n$ is in convex position and $|V| = m$.
\end{enumerate}
\end{lemma}

\begin{proof}
We would like to enlarge $P$ until each facet contains a unique integer point.  To make these points unique, we will perturb the facets slightly.  For this process, we first establish a bounded region to work on.  

Let $Q$ be any full-dimensional polytope containing $X$ in its interior such that $P$ is non-redundant with respect to $Q \cap \Z^n$ and $P \cap Q \neq \emptyset$.  For $i=1, \dots, m$, let $Y^{>}_i = Q \cap \{x \st a_i \cdot x > b_i\} \cap \Z^n$, which is non-empty since $P$ is non-redundant with respect to $Q \cap \Z^n$, and let $Y_i^{\leq} = Q \cap \{x \st a_i \cdot x \leq b_i\} \cap \Z^n$.

Let $\epsilon >0$ such that for all $i=1, \dots, m$,
\begin{equation}
\label{eq:small-epsilon}
a_i \cdot y > b_i + 3 \epsilon \text{ for all } y \in Y^{>}_i.
\end{equation}
Next define perturbation vectors $\bar a_i$ such that the matrix $A'$ with rows $a_i'=a_i + \bar a_i$ is generic and has all entries linearly independent over $\Q$ and such that 
\begin{equation}
\label{eq:small-bar-a}
| \bar a_i \cdot  x| < \epsilon  \text{ for all } x \in Q \text{ and  } i=1, \dots, m.
\end{equation}
Such perturbations $\bar a_i$ exist due to Lemma~\ref{lem:q-linear-independence}.  
Then, for all $i=1, \dots, m$,
\begin{align}
\label{eq:vj}
a'_i \cdot y &\leq a_i \cdot y + |\bar a_i \cdot y| < b_i + \epsilon & \text{ for all } y \in Y^{\leq}_i,\\
\label{eq:vi}
a'_i \cdot y &\geq a_i \cdot y - |a'_i \cdot y| > b_i + 3\epsilon - \epsilon = b_i +  2\epsilon & \text{ for all } y \in Y^{>}_i.
\end{align}
Set $P_0' = \{ x \st  a_i' \cdot x \leq b_i + 2\epsilon, i=1, \dots, m\}$.  Then $P_0'$ satisfies $\intr( P_0') \cap Q \cap \Z^n = P_0' \cap Q \cap \Z^n=X$ and also $P_0'$ is non-redundant with respect to $Q \cap \Z^n$.

We want that $P_0'$ does not contain integer points outside of $Q$.  We claim that $P_0' \cap \Z^n = X$.  Suppose this is not true.  Consider a non-redundant subsystem of the inequalities from $P_0' \cap Q$.  This system contains all inequalities from $P_0'$ and also at least one from $Q$.  But this contradicts the assumption that $P$ was maximum non-redundant.

We next claim that $P_0'$ is full-dimensional and bounded.  By choice of $Q$, we have $P \cap Q \neq \emptyset$. That is, there exists a $\bar x \in P \cap Q$.  Then, by \eqref{eq:small-bar-a} and the fact that $\bar x \in P$, we have $a_i' \cdot \bar x \leq a_i \cdot \bar x + |\bar a_i \cdot \bar x| \leq b_i + \epsilon < b_i + 2 \epsilon$, and hence $\bar x \in \intr(P_0')$.  Thus $P_0'$ is full-dimensional.  Suppose for the sake of contradiction that $P_0'$ is not bounded.  Then there exists a non-zero vector $u \in \rec(P_0') = \{x \st a'_i\cdot x \leq 0, i=1, \dots, m\}$, where $\rec$ denotes the recession cone.  By Lemma~\ref{lem:cone-non-empty}, since $A'$ is generic, $\rec(P_0')$ is full-dimensional.  Since both $P_0'$ and $\rec(P_0')$ are full-dimensional, $P_0'$ contains infinitely many integer points.  This contradicts the fact that $P_0' \cap \Z^n = X$ and the assumption that $X$ is finite.

We now relax the inequalities of $P_0'$ until they reach integer points in the following way.
In sequence, for $i=1, \dots, m$,  set $b_i' = a_i' \cdot  v_i $ for $v_i \in \argmin_{x \in  P_i' \cap \Z^n}   a_i' \cdot x$  where $P'_i$ is the polyhedron given by the inequalities
\begin{align}
 a_j' \cdot x &\leq b_j'  & j=1, \dots, i-1,\\
 \label{eq:9}
 a_j' \cdot x &\geq b_j  + 2\epsilon & j=i,\\
 \label{eq:10}
 a_j' \cdot x &\leq b_j + \epsilon& j=i+1, \dots, m. 
\end{align}
Since $P_0'$ is non-redundant with respect to $Q \cap \Z^n$, and hence also with respect to $\Z^n$, such a $v_i$ exists for each $i=1, \dots, m$.

Define $P' = \{x \st   a_i' \cdot x \leq b_i', \, i=1, \dots, m\}$.  For any $i = 1, \dots, m-1$ and $j = i+1, \dots, m$, since $v_i \in P'_i$, 
by~\eqref{eq:10} for $P'_i$,
\begin{equation}
a'_j \cdot v_i \leq b_j + \epsilon < b_j + 2\epsilon.
\end{equation}
Then by \eqref{eq:9} for $P'_j$, $v_i \notin P'_{j}$.    
Hence, 
$v_1, \dots, v_{i} \notin P'_{i+1}$.  Thus, $v_{i+1} \notin \{v_1, \dots, v_i\}$, and therefore all $v_1, \dots, v_m$ are distinct.  Since each $a'_i$ has coordinates that are linearly independent over $\mathbb{Q}$,  each facet of $P'$ can contain at most one integer point.
Furthermore, no new integer points were introduced to the interior of $P'$, so $\intr(P') \cap \Z^n = X$.
Hence,
\begin{equation*}
V \coloneqq (P'\setminus P) \cap \mathbb{Z}^n = \{v_1,\ldots, v_m\}\enspace.
\end{equation*}
Moreover, since there are $m$ facets containing $m$ integer points, each facet contains exactly one integer point, and hence this point is on the relative interior of the facet, which also implies that the points in $V$ are in convex position.

This completes the proof.
\end{proof}

\section{Asymptotic Upper Bounds}\label{sec: upper bounds}

We obtain asymptotically sublinear (in $k$) upper bounds on $c(n,k)$ by focusing on a question based on the lattice polytopes.
\begin{definition}
Given $n,k$ two nonnegative integers, $\ell(n,k)$ is the smallest integer such that for any set $V \subseteq \Z^n$ in convex position with $|V|=\ell(n,k)$, we have that $|\iconv(V) \setminus V| \geq k$. 
\end{definition}
 Alternatively, $\ell(n,k+1) - 1$ is the maximum number of vertices of a lattice polytope in $\R^n$ having at most $k$ non-vertex lattice points.  
 Following the definition, for any $V\subseteq\Z^n$ is in convex position with $|V|>\ell(n,k)$, we have $|\iconv(V)\setminus V|\geq k$.  A more general result was also recently proved in~\cite[Theorem 2]{averkov2016} that implies the bound.  We provide a proof here for completeness of the inequality that we need.

\begin{lemma} 
\label{lem:ell-bound}
$c(n,k) \leq \ell(n,k+1)-1.$
\end{lemma}
\begin{proof}
Suppose, for the sake of contradiction,  that $c(n,k) \geq \ell(n,k+1)$.  
Let $P = \{x \st  Ax \leq b\}$ be a non-redundant system of $c(n,k)$ inequalities with $k$ integer solutions, and let $V\subseteq \Z^n$ in convex position with $|V| = c(n,k)$ be given by Lemma~\ref{lem: one point per facet}.
Since $V$ is a set of at least $\ell(n,k+1)$ points in convex position, $|\iconv(V)\setminus V|\geq k+1$.
But, $\iconv(V)\setminus V \subseteq P\cap\Z^n$, so we must have $|\iconv(V)\setminus V| \leq k$, which is the contradiction.
\end{proof}

Clearly, the points $\{0,1\}^n$ demonstrate that $\ell(n,1) > 2^n$.  
The proof of $\ell(n, 1) \leq 2^n + 1$ follows by the parity and pigeonhole principle argument outlined in the introduction.  
Consider the parities of the $2^n+1$ points in $V$.  Since there are only $2^n$ total parities, there must be two points with the same parity in $V$.  Since these two points have the same parity, the midpoint of these two points is also an integer point, and thus $V$ contains at least one integer point in its convex hull that is not in $V$.  This is a classic proof that was known in~\cite{doignon1973, bell1977, scarf1977}.

For any set $V \subseteq \R^n$, let $M(V)$ denote  the \emph{set of midpoints} of points in $V$ that are not already part of $V$, that is, $M(V) =  (\tfrac{1}{2}(V + V)) \setminus V = \{ \tfrac{1}{2}(x_1 + x_2) \st x_1, x_2 \in V\}\setminus V$.  
Given $n,s$ two nonnegative integers, $\mu(n,s)$ is the minimum number of midpoints $|M(V)|$ for any set $V \subseteq \R^n$ with $|V| = s$ such that no three points of $V$ are collinear.
In a similar way, $\muconvex(n,s)$ is the minimum number of midpoints $|M(V)|$ for any set $V \subseteq \R^n$ with $|V| = s$ such that $V$ is in convex position.
Obviously, $\mu(n,s)\leq \muconvex(n,s)$; $\ell$ and $\muconvex$ are related in the following way, which follows as well by a parity and pigeonhole principle argument.
The key point is that lower bounds on $\mu$ lead to upper bounds on $\ell$ and therefore upper bounds on $c$, which is what we are after.
\begin{lemma}
\label{lem:mu-ell-bound}
Let $n\geq 1$, $k\geq 0$, and let $ s_{n,k} \coloneqq \min\{s \st \muconvex(n,s) \geq k\}$.  Then  
$\ell(n,k)   \leq   (s_{n,k}-1) 2^n+1$.
\end{lemma}  
\begin{proof}
Using the parity argument with pigeonholes, given $(s_{n,k}-1) 2^n+1$ integer points in convex position, there is a parity class containing at least $s_{n,k}$ of these points.  Of these points, there are at least $k$ midpoints which are all integral.  Thus, any set of  $(s_{n,k}-1) 2^n+1$ integer points in convex position contains at least $k$ other integer points in their convex hull.
\end{proof}

For dimension $n=2$, Erd\H{o}s, Fishburn, and F\"uredi~\cite[Theorem 1]{erdos91} show that $\muconvex(2,s) = \Theta(s^2)$, which leads us to conclude $c(2,k)=O(\sqrt{k})$ (although Theorem~\ref{thm:c(2 k)} does better).  It appears that the quantity $\muconvex(n,k)$ has not been studied for $n \geq 3$.  
However, the quadratic growth of $\muconvex(n,s)$ that is observed in dimension $n=2$ does not hold in higher dimensions.  Indeed, in Proposition~\ref{prop:muconvex-bound}, we show that $\muconvex(n,s) = O(s^{(n+1)/(n-1)})$.

In this section we are interested in lower bounds on $\muconvex(n,s)$.  
It is trivial to show $\muconvex(n,s)\geq s-1$.  Some superlinear bounds come directly from bounds for $\mu(n,s)$.  
Pach~\cite{pach2003} and Stanchescu~\cite{stanchescu2002} have both proved superlinear lower bounds for $\mu(2,s)$, and we could take advantage of these by projecting a convex set in $n$ dimensions to a set in the plane that has no three term arithmetic progression.  Sanders~\cite{sanders2010} proved a more general result for abelian groups and any finite subset containing no nontrivial three-term arithmetic progressions that implies better bounds for $\mu(n,s)$. Most recently, Henriot~\cite{henriot} improved upon Sanders' bounds.  Henriot shows that there exists a constant $C>0$ such for any abelian group $G$ and finite subset $A\subseteq G$ with $|A| \geq 3$ and containing no nontrivial three-term arithmetic progressions we have that
\[|A+A|\geq C|A|\left(\frac{\log{|A|}}{(\log\log|A|)^7}\right).\]
This implies that there exists a constant $C' >0$ for large enough $s$ that 
$$
\mu(n,s) \geq C's \left(\frac{\log{s}}{(\log\log s)^7}\right).
$$

Even so, we can do better for $\muconvex(n,s)$ by applying the following result of Schoen and Shkredov~\cite{Schoen2014}.
\begin{theorem}
\label{thm:schoen}
Let $N$ be a positive integer.  Suppose that $A \subseteq \{1, \dots, N\}$ has no solution to 
\begin{equation}
\label{eq:5}
\frac{1}{5}x_1 + \frac{1}{5}x_2 + \frac{1}{5}x_3 + \frac{1}{5}x_4 + \frac{1}{5}x_5 = x_6
\end{equation}
for distinct integers $x_1, \dots, x_6 \in A$.  Then 
\begin{equation}\label{eq:freiman lower bound}
|A| = \frac{N}{e^{\Omega(\log^{1/7} N)}}.
\end{equation}
\end{theorem}

In order to apply this theorem for our purposes, we must transfer the bounds for subsets of $\{1,\dots, N\}$ to the space we are interested in.   In particular, this must be done in a way that allows $N$ to be of the order of $|A|$ where $A \subseteq \Z^n$ is in convex position.
\begin{definition}[{\cite[Definition 5.21]{TV06}}]
 Let $k \geq 1$ and let $A$ and $B$ be sets in abelian groups $Z$ and $W$ respectively. A \emph{Freiman homomorphism} of order $k$ from $A$ to $B$ is a map $\phi\colon A \to B$ with the property
\begin{equation}
a_1 + \cdots + a_k = a'_1 + \cdots + a'_k \ \Rightarrow \ 
\phi(a_1) + \cdots + \phi(a_k) = \phi(a'_1) + \cdots + \phi(a'_k)
\end{equation}
for all $a_1,\dots,a_k$, $a'_1,\dots, a'_k \in A$. If, in addition, there is an inverse map $\phi^{-1} \colon B \to A$ which
is also a Freiman homomorphism of order $k$ then $\phi$ is a \emph{Freiman isomorphism} of order $k$.
\end{definition}
We will use the term \emph{Freiman $k$-isomorphic} as shorthand for Freiman isomorphic of order $k$.
See~\cite[Section 5.3]{TV06} for a discussion on Freiman isomorphisms.  
A Freiman isomorphism $\phi\colon A\to B$ is useful for transferring results in the domain $B$ to the domain $A$.
In what follows we will take $A\subseteq \Z^n$ to be a set of points in convex position and $B=\Z$, and transfer Schoen and Shkredov's Theorem from $\Z$ to $\Z^n$.
Specifically, \eqref{eq:freiman lower bound} gives us a lower bound on $|A|=|\phi(A)|$.
A technical problem arises because we need $\phi(A)\subseteq\{1,2,\ldots,N\}$, for some $N$ not too much larger than $|A|$, in order to apply Theorem~\ref{thm:schoen}.
The next two lemmas help us construct a satisfactory Freiman isomorphism.

\begin{lemma}[{\cite[Lemma 5.25]{TV06}} ]
\label{lemma:525}
Let $A$ be a finite subset of a torsion-free abelian group $Z$. Then for any integer $k$, there is a Freiman isomorphism $\phi \colon A \to B$ of 
order $k$  to some finite subset $B$ of the integers $\Z$.
\end{lemma}
The following theorem by Ruzsa allows us to control the size of the ambient set by  switching to a smaller subset.

For $A,B \subseteq \R^n$, and $k$ a positive integer, define $kA \coloneqq A + \cdots + A$ (with $k$ summands) and $A -B \coloneqq \{ a - b \st a \in A, b \in B\}$. 
\begin{theorem}[{\cite[Theorem 2]{Ruz92}}]
\label{theorem:ruz2}
Let $A$ be a finite set of integers, $|A| =  N$ and $k \geq 2$ an integer. 
There is a set $A' \subseteq A$ with $|A'| \geq N/k^2$ which is Freiman isomorphic of order $k$
to a set $T \subseteq \{1, \dots, 2|kA - kA|\}$.
\end{theorem}
From these tools, we can conclude the following proposition.

\begin{proposition}
\label{prop:sum17}
There exists a constant $C>0$ such that for any finite $A \subseteq \R^n$ containg no solutions to \eqref{eq:5} with $x_1,x_2,\ldots,x_6$ distinct, we have
$$| A + A| \geq |A| e^{C\log^{1/7} |A|}.$$
Thus, we have that $\muconvex(n,s) =  s \,e^{\Omega(\log^{1/7}s)}$.
\end{proposition}
\begin{proof}
Using Lemma~\ref{lemma:525} and Theorem~\ref{theorem:ruz2}, we see that there is a subset of $A$ that is Freiman $5$-isomorphic to a subset $A'$ of $\{1, \dots, 2 |5A  - 5A|\}$ of size at least $|A|/25$.   This follows from the fact that the composition of  Freiman $k$-isomorphisms is a Freiman $k$-isomorphism.   Since, trivially, $|A+A| \leq \frac{ |A+A|}{|A|}  |A|$, the Pl{\"u}nnecke-Ruzsa estimates~\cite[Corollary 6.29]{TV06} imply that 
$$
|5A - 5A| \leq \left( \frac{|A + A|}{|A|}\right)^{10} |A|.
$$
As mentioned before, since $A$ contains no solutions to \eqref{eq:5}  with distinct numbers, the set $A'$ contains no solutions to \eqref{eq:5} with no distinct integers. By Theorem~\ref{thm:schoen} we then see
that
\begin{align*}
\frac{|A|}{25} \leq |A'| &\leq \frac{2|5A - 5A|} {e^{ C'\log^{1/7} 2|5A - 5A|}}\\
&\leq \frac{\left( \frac{|A + A|}{|A|}\right)^{10} |A|}{ e^{C'\log^{1/7} 2|5A - 5A|}}\\
&\leq \frac{\left( \frac{|A + A|}{|A|}\right)^{10} |A|}{ e^{C'\log^{1/7} |A|}}.
\end{align*}
Here $C'$ is the constant from Theorem~\ref{thm:schoen}, and the last inequality follows from the fact that $|A| \leq |5A - 5A|$.  
Rearranging this gives the claimed proposition. 
\end{proof}

We are in position to prove the sublinear upper bound on $c(n,k)$.
\begin{theorem}\label{thm: sublinear upper bound}
For all $n,k\geq 1$,
$$
c(n,k) \leq \ell(n,k+1) \leq \frac{k}{e^{\Omega(\log^{1/7} k)}} \cdot 2^n.
$$
\end{theorem}
\begin{proof}
Combining Lemma~\ref{lem:ell-bound} and Lemma~\ref{lem:mu-ell-bound}, we have that $c(n,k) \leq \ell(n,k+1)-1 \leq (s_{n,k+1} - 1)2^n$, where
\begin{align*}
s_{n,k+1} &= \min\{ s \st \muconvex(n,s)\geq k+1\} \\
&\leq \min \{ s \st s e^{ \Omega(\log^{1/7} s)} \geq k+1\}\\
&\leq \frac{k}{e^{\Omega(\log^{1/7} k)}},
\end{align*}
where the second to last inequality follows from Proposition~\ref{prop:sum17}.
\end{proof}

As pointed out by Bell~\cite{bell1977}, in dimension $n=2$, we can do much better.  Here is a proof based on a recent bound for the minimum area of a lattice $n$-gon that gives an explicit upper bound.
\begin{theorem}
\label{thm:c(2 k)}
$$
c(2,k) \leq \ell(2,k+1) \leq \lfloor{4.43 (k+4)^{\frac{1}{3}}}\rfloor.
$$
\end{theorem}
\begin{proof}
Let $\ell =  4.43 (k+4)^{\frac{1}{3}}$.  We want to show that for any set $V \subseteq \Z^2$ in convex position, if $|V| \geq \ell$, then $|\iconv(V) \setminus V| \geq k+1$.  We will show the contrapositive: if $|\iconv(V) \setminus V| < k +1$ then $|V| < \ell$.   

Let $P = \conv(V)$.  
Let $v = |V|$ be the number of vertices of $P$, let $b$ be the number of lattice points on the boundary of $P$ that are not vertices, and let $i$ be the number of interior lattice points.  Therefore $i+b < k+1$.  By Pick's Theorem, the area $A$ of $P$ is given by $A = i + \frac{v + b}{2} - 1$.  
 By~\cite{rabinowitz1993}, $A \geq \frac{v^3}{8 \pi^2}$.  Hence
 $
  \frac{v^3}{8 \pi^2} \leq i + \frac{v + b}{2} - 1.
 $
After rearranging, we have
$
  \frac{v^3}{8 \pi^2}  - \frac{v}{2} \leq i + \frac{b}{2} - 1 < k.
$
One choice of a lower bound for $v \geq 0$ is $\frac{v^3}{4.43^3} - 4 \leq \frac{v^3}{8 \pi^2}  - \frac{v}{2} < k$.  This lower bound is most conveniently verified graphically and holds for $v \geq 0$.  The result follows now by rearranging to $v^3 < (4.43)^3 (k+4)$, applying a cube-root to both sides, and applying the floor operator since $v$ is a nonnegative integer.
\end{proof}
From the calculation above, we could also say that for every $\epsilon >0$ there exists a constant $C_\epsilon$ such that $\ell(n,k+1) \leq (8 \pi^2 + \epsilon)^\frac{1}{3} (k + C_\epsilon)^{\frac{1}{3}}$.  It was shown that for $A_v$, the minimum area of a convex lattice polygon with $v$ vertices, the limit $\lim_{v \to \infty} \tfrac{A_v}{v^3}$ and is suggested that the limit most likely very close to $0.0185067$~\cite[Page 2]{barany2004}.  Using this asymptotic result, the asymptotic behavior of $(8 \pi^2 + \epsilon)^\frac{1}{3} k^{\frac{1}{3}}$ could be improved further by a multiplicative factor.

\subsection{Tighter bounds for specified arrangements}
The dependence on $k$ of the upper bound is tight for $n=2$, but it is not tight for $n\geq 3$ as we know from the aforementioned bounds of Averkov et al.~\cite{averkov2016}.
The goal of this section is to achieve a smaller upper bound for specified arrangements of points.
For a set $X\subseteq \Z^n$, let $c(X)$ denote the smallest upper bound on the number of inequalities in any non-redundant system $Ax \leq b$ such that for $P=\{x\in\R^n\st Ax\leq b\}$ we have $P\cap \Z^n=X$.  That is,
\begin{equation}
\begin{aligned}
c(X) = \sup\{m \st & P \cap \Z^n = X \text{ for any } P = \{x \in \R^n \st Ax \leq b\}\\
&\text{with } A \in \R^{m\times n}, b \in \R^m, \text{ and } Ax\leq b \text{ non-redundant}\}.
\end{aligned}
\end{equation}
If no such system exists, then we define $c(X)=-\infty$. 
One can show that $c(X)$ is no larger than the Helly number of $S = \Z^n \setminus X$~\cite{averkov2012}, but these numbers do not always coincide. For instance, for $X = \Z \times\{0\} \subseteq \Z^2$, we have $c(X) = 2$, but the Helly number of $S = \Z^2 \setminus X$ is 6.

The quantities $c(X)$ and $c(n,k)$ are related as follows.
\begin{proposition}\label{prop: c(S) vs c(n,k)}
\[c(n,k) = \max_{\substack{X\subseteq\Z^n\\|X|=k}}c(X),\]
in particular $c(X)$ is finite whenever $X$ is finite.
\end{proposition}
If we allow $X$ to be infinite, then we can have $c(X) = \infty$.  For instance, let $X = \{ (x,y) \in \Z^2 :  \tfrac{1}{2} \leq y, \sqrt{2} y \leq x\}$.

\renewcommand{\d}{\mathrm{rc}}

The two main theorems of this section are Theorem~\ref{thm: subspaces} and Theorem~\ref{thm: 2d S-critical bound}.
The values for both bounds can be substantially lower than $c(n,k)$.
The first theorem says that if $X$, with $|X|=k$, lies in a $d$-dimensional subspace, then $c(X)\leq c(d,|X|)+2(2^n-1)$.  
This upper bound is sharp, and equal to $2(2^n-1)$, for any set of collinear points, which gives an alternative proof that $c(n,2)\leq 2(2^n-1)$ as every two points are collinear.
The bound for the case of collinear points is used in our proof that $c(2,5)\leq 7$.

The second theorem, Theorem~\ref{thm: 2d S-critical bound},  implies that for $X \subseteq \Z^n$ finite, $c(X)\leq \d(X)(2^n-1)$, where $\d(X)$ denotes the minimum number of inequalities in any system whose integer solutions are exactly $X$, or $-\infty$ if there is no such system.  The quantity $\d(X)$, known as the \emph{relaxation complexity} of $X$, was studied in~\cite{Kaibel2015}.  A similar quantity was studied earlier in~\cite{JEROSLOW1975119}.

From the definition of $\d(X)$, we have $\d(X) \leq c(X)$.  Thus, for any $n\geq 1$, $k\geq 0$, we can bound $c(n,k)$ as 
\[\max_{|X|=k}\d(X)\leq \max_{|X|=k}c(X)=c(n,k)\leq (2^{n}-1)\max_{|X|=k}\d(X).\] 

\begin{theorem}\label{thm: subspaces}
Let $1\leq d\leq n$ and $X\subseteq\Z^n$. If $\iconv(X)=X$ and $X$ is contained in a $d$-dimensional affine subspace of $\R^n$, then $c(X) \leq c(d,|X|) + 2(2^n-2)$.
\end{theorem}

\begin{proof}
  Let $k=|X|$.
  The Doignon-Bell-Scarf Theorem covers the case $k=0$ and \eqref{eq: qdbs bound} covers $k=1$.  
Henceforth, let $k\geq 2$. 
Without loss of generality we may assume that $0\in X$, and hence any affine subspace containing $X$ is in fact a linear subspace.

Suppose that $P=\{x\in\R^n\st Ax\leq b\}$ is a polytope defined by a maximum non-redundant system of $c(X)$ inequalities with $P\cap\Z^n=X$ and let $V$ be given by Lemma~\ref{lem: one point per facet}.
Let $L\subseteq \R^n$ be a $d$-dimensional rational linear subspace containing $X$.
There exists a hyperplane $H$ such that $L\subseteq H$ and $H\cap(\Z^n\setminus L)=\emptyset$, since a generic hyperplane containing $L$ has this property.
Let $H_0$ and $H_1$ denote the closed halfspaces with $H_0\cap H_1=H$.
Notice that no two points in $V\cap \intr(H_0)$ can have the same parity nor can a point here have the same parity as a point in $X$ because their midpoint would not lie in $L$.
The same holds for $V\cap\intr(H_1)$.
Therefore, if $p$ is the number of distinct parities of points in $X$ then both $|V\cap \intr(H_0)|$ and $|V\cap \intr(H_1)|$ are at most $2^n-p$.

The points in $X$ have at least two distinct parities, i.e.~$p\geq 2$.
To see this, suppose $a,b\in X$ are distinct points of the same parity.
By assumption $\iconv(X)=X$, so every integer point on the line segment connecting $a$ and $b$ also lies in $X$.  Neighboring points on the line segment cannot be of the same parity, so there is at least one additional point on the segment that has a different parity from $a$ and $b$.
We have established that both $|V\cap \intr(H_0)|$ and $|V\cap \intr(H_1)|$ are at most $2^n-2$.

It remains to bound $|V\cap H|$.
Our condition that $H\cap(\Z^n\setminus L)=\emptyset$ implies that $|V\cap H|=|V\cap L|$.
Since $L$ is a $d$-dimensional rational linear subspace, one can identify $L$ with $\R^d$ and $L\cap\Z^n$ with $\Z^d$.
The number of non-redundant facets of $P\cap L$, seen as a polytope in $L$, is at most $c(d,k)$, by definition of $c(d,k)$ and at least $|V\cap L|$, by our application above of Lemma~\ref{lem: one point per facet}.
Thus we have,
\[c(X)=|V| = |V\cap L|+|V\cap \intr(H_0)| + |V\cap\intr(H_1)| \leq c(d,k) + 2(2^n-2).\]
\end{proof}
As any $k$ points in $\R^n$ are contained in a $(k-1)$-dimensional affine subspace, a consequence of Theorem~\ref{thm: subspaces} is the following corollary.

\begin{corollary}\label{cor: large n}
If $1\leq k\leq n$, then $2^{n+1} - k\leq c(n,k)\leq c(k-1,k)+2(2^n-2)$.  
\end{corollary}

The corollary comes with a lower bound, too, because $c(n,k)\geq c(n,k-1)-1$ by Lemma~\ref{lem:inductive}.
In particular, $|c(n,k)-2^{n+1}|= O(1)$, for each fixed $k$.

\begin{theorem}\label{thm: 2d S-critical bound}
Let $X\subseteq \Z^n$ be a set of points such that $X = \iconv(X)$ and let $Q=\{x\in\R^n\st a_i \cdot x\leq b_i, i=1, \dots, q\}$ for some $a_1, \dots, a_q \in \R^n$, $b \in \R^q$ be any polyhedron with $Q\cap \Z^n=X$.
Then $c(X)\leq q(2^n-1)$.  
\end{theorem} 
\begin{proof}
The statement is implied by the Doignon-Bell-Scarf Theorem if $X=\emptyset$, so suppose $X\neq\emptyset$.
Let $P$ be a maximum non-redundant polyhedron such that $P \cap \Z^n = X$ and let $P'$ and $V$ be given by applying Lemma~\ref{lem: one point per facet} to $P$.  Then $V \subseteq \Z^n$ with $|V| = c(X)$ and each facet of $P'$ contains a unique $v \in V$ in its relative interior.  Then $X \cap V = \emptyset$, and hence $Q \cap V = \emptyset$.  Thus, $V \subseteq (\R^n \setminus Q)$.  Since $\R^n \setminus Q = \cup_{i=1, \dots, q} \{x \st a_i \cdot x > b_i\}$, we have that 
$$
V = \bigcup_{i=1}^q V_i, \ \text{ where } \ V_i = V \cap \{x \st a_i \cdot x > b_i\} \text{ for } i=1, \dots, q.
$$
For each $i=1, \dots, q$, let $u_i \in \argmax\{ a_i \cdot x \st a_i \cdot x \leq b_i, x \in V \cup X\}$.  Since $X\neq\emptyset$ and $X \subseteq Q$, there exists such a $u_i$.  Furthermore, by choice of $u_i$ and $V_i$, we have that $\bar V_i = V_i \cup \{u_i\}$ is in convex position and $\iconv(\bar V_i) \setminus \bar V_i = \emptyset$.  This is because $V_i$ is a subset of the integer points on the boundary of $P'$ and $u_i \in P'$ chosen to not be contained in $\iconv(V_i)$.  But since $\ell(n,1) = 2^n + 1$, we must have $|\bar V_i| \leq 2^n $ and hence $|V_i| \leq 2^n -1$ for each $i=1, \dots, q$.  Finally,
$$
c(X) = |V| = \left|\bigcup_{i=1}^q V_i\right| \leq \sum_{i=1}^q |V_i| \leq q(2^n -1).
$$
\end{proof}

\section{Asymptotic Lower Bounds}\label{sec: lower bounds}

In this section, we prove that $c(n,k)$ is at least as large as the maximum number of vertices of any lattice polytope with $k$ non-vertex integer points.  This result was also shown recently in~\cite[Theorem 2]{averkov2016}.
It follows that exhibiting such a polytope with many vertices gives a lower bound for $c(n,k)$. 
We use the integer hull of the $n$-dimensional ball.
The result is a lower bound for $c(n,k)$ whenever there exists $r>0$ so that $k$ is the number of non-vertex lattice points in the integer hull of the ball with radius $r$.
We do not know if every $k\in\N$ can be achieved in this way, for example by a translation of the ball, so this leads us to a lower bound on $c(n,k_r)$ only for a sequence of values $k_r\to\infty$. 
The problem is compounded by the fact that $c(n,k)$ may decrease as $k$ increases, which makes it difficult to extend the bound from the sequence $k_r$ to all $k\geq 0$.
Fortunately, as the next lemma shows, the decrease is modest.
This is enough to fill in the gaps between consecutive values $k_r$.

The following lemma is implied by~\cite[Theorem 3]{averkov2016} applied for $S = \Z^n$.  
\label{sec:lower-bounds}
\begin{lemma}
\label{lem:inductive}
For all $k,n\geq1$, $c(n,k)\geq c(n,k-1)-1$.
\end{lemma}
\begin{proof}
Let $P=\{x\in\R^n\st Ax\leq b\}$ be maximum non-redundant polyhedron with $c(n,k-1)$ inequalities and containing $k-1$ lattice points.  Let $P' = \{ x \in \R^n \st A'x \leq b'\}$ the polytope as given in Lemma~\ref{lem: one point per facet}.  Then $\intr(P') \cap\Z^n = k-1$ and each of the $c(n,k-1)$ facets of $P'$ has an integer point on its relative interior.  Thus, there exists an $\epsilon >0$ such that $P'' = \{ x \in \R^n \st a_1' \cdot x \leq b'_1, a_i' \cdot x \leq b_i'-\epsilon \text{ for all } i =2, \dots, c(n,k-1)\}$ contains exactly $k$ lattice points.  By construction, there are at least $c(n,k-1) -1$ inequalities, those numbered $2,3,\ldots, c(n,k-1)$, that cannot be removed from $P''$ without increasing the number of integer points in $P''$.  Hence, $c(n,k) \geq c(n,k-1) -1$.
\end{proof}

As with the upper bound, our lower bound is proved by considering integer points in convex position.
\begin{definition} 
Let
\begin{equation}\label{eq: alpha(n,k)}
\alpha(n,k) \coloneqq \max\left\{|V| \st V\subseteq \mathbb{Z}^n\text{ in convex position},
 |\iconv(V)\setminus V|=k\right\}.
\end{equation}
\end{definition}
In other words, $\alpha(n,k)$ is the maximum number of vertices of a lattice polytope that contains exactly $k$ non-vertex lattice points.  As mentioned before, an alternative definition of $\ell(n,k)$ is that $\ell(n,k+1)-1$ is the maximum number of vertices of a lattice polytope $P \subseteq \R^n$ having at most $k$ non-vertex lattice points.  Thus, 
$$
\ell(n,k+1) - 1 = \max_{i=0, \dots, k} g(n,k).
$$
The quantity $g(n,k)$ was studied in~\cite{averkov2016} denoted as $g(\Z^n, k)$.  

A related quantity has been studied by Averkov~\cite{averkov2013}.  Let $S$ be a discrete subset of $\R^n$ and let $g(S)$ denote the maximum number of vertices of a polytope $P$ with vertices in $S$ 
such that $P$ contains no other points in $S$. The number $g(S)$ is equal to the Helly number of $S$ since $S$ is discrete~\cite[Lemma 2.2]{deloera2015}.  See also~\cite{amenta2015helly, hoffman79, averkov2013}.  Clearly $\alpha(n,k)  \leq \max\{ g(\Z^n \setminus V) \st V \subseteq \Z^n,  |V| = k,  \iconv(V)  = V\}$, but it is unclear if equality always holds. 

Some lower bounds for $\alpha(2,k)$ for specific values of $k$ can be directly derived from examples in~\cite{castryck2012}.  Furthermore,~\cite{averkov2016} uses the database available in~\cite{castryck2012} to exactly compute values for $\alpha(2,k)$ for all $k \in \{0,\dots, 30\}$.  Castryck~\cite{castryck2012} studies 
 the number of interior lattice points for convex lattice $n$-gons in the plane.   They define the genus $\mathbb{g}(n)$ as the minimum number of interior lattice points of a lattice $n$-gon.  They compute $\mathbb{g}(n)$ for $n=1, \dots, 30$ and also provide related information such as the number of equivalence classes up to lattice invariant transformations.  Although the computations done by Castryck are with respect to interior lattice points, many of their examples have no boundary lattice points that are not vertices.  Therefore, these examples directly provide lower bounds on the values of $\alpha(2,k)$ and hence $c(2,k)$.  
For example, 
$\alpha(2,17) \geq 11$, $\alpha(2,45) \geq 15$, $\alpha(2,72) \geq 17$, and $\alpha(2,105) \geq 19$.

Our strategy is to prove that $c(n,k)\geq\alpha(n,k)$ and then use known properties of the integer hull of the Euclidean ball to prove a lower bound on $\alpha(n,k)$.
The first lemma in this direction is the following.  

Let $\| \cdot \|$ denote the standard Euclidean norm.

\begin{lemma}\label{lem: alpha P}
Let $k\in\Z_{\geq0}$ and $V\subseteq \Z^n$ be a maximizer of~\eqref{eq: alpha(n,k)}.
There exists a polyhedron $P=\{x\in\R^n\st Ax\leq b\}$ with $|V|$ facets such that \textup{(a)} each point of $V$ is contained in the relative interior of a different facet of $P$ and \textup{(b)} $P\cap\Z^n = \iconv(V)$.
\end{lemma}
\begin{proof}
Let $V = \{v_1, \dots, v_k\}$.  Because the points in $V$ are in convex position, for each $v_i\in V$ there exists $a_i\in\R^n$ and $b_i\in \R$ such that $a_i \cdot v_i = b_i$ and, for all $y\in V\setminus\{v_i\}$, $a_i \cdot y<b_i$.
Let $Ax \leq b$ be the system of $k$ inequalities $a_i \cdot x\leq b_i$ with $i\in \{1, \dots, k\}$, and let $P = \{x\in\R^n\st Ax\leq b\}$.
By construction $P$ has $|V|$ facets and satisfies (a).

Furthermore, $\conv(V)\subseteq P$, so $\iconv(V)\subseteq P\cap\Z^n$, as well.
Suppose, for contradiction, that $P$ does not satisfy (b), so that there exists an integral point in~$P$ that is not in the convex hull of $V$.
Let $x\in\Z^n$ be such a point that minimizes $\dist(x,\conv(V))\coloneqq \min_{s\in \conv(V)}\|x-s\|^2$.
Notice that $V\cup\{x\}$ is in convex position because we chose $x\notin\conv(V)$ and each point of $V$ is a unique point of $V$ in the relative interior of some facet of $P$.
It suffices for us to show that $\iconv(V)\setminus V = \iconv(V\cup\{x\})\setminus (V\cup\{x\})$, because this contradicts the maximality of $V$.

Indeed, suppose that there is an integral point $y\in P\setminus\conv(V)$ and some $s\in\conv(V)$ such that $y=\lambda x + (1-\lambda)s$, $0<\lambda<1$.
The function $\dist(\cdot,\conv(V))$ is a convex function on $\R^n$, so $\dist(y,\conv(V))\leq \lambda \dist(x,\conv(V)) < \dist(x,\conv(V))$, which contradicts the choice of $x$. 
Thus, $\iconv(V\cup\{x\})\setminus (V\cup\{x\}) = \iconv(V)\setminus V$, which completes the proof.
\end{proof}

Let $\vec{\bf t}$ denote the vector with all entries $t$.    The following lemma is implied by the recent result ~\cite[Theorem 2]{averkov2016}.

\begin{lemma}\label{lem:albc}
For all $n\geq1$ and $k\geq 0$, $c(n,k) \geq \alpha(n,k)$.
\end{lemma}
\begin{proof}
Let $V$ be a maximizer of~\eqref{eq: alpha(n,k)} and let $P=\{x\in\R^n \st Ax\leq b\}$ be the polyhedron from Lemma~\ref{lem: alpha P}. 
We have $\iconv(V)\setminus V \subseteq \intr(P)$. 
Thus, there exists $\epsilon>0$ such that, for $b'=b-\epsilon\vec{\bf 1}$, the polyhedron $P' = \{x\in\R^n \st  Ax\leq b'\}$ has $\iconv(V)\setminus V=P'\cap\Z^n$ and removing any inequality from $Ax\leq b'$ adds a point from $V$ to the polyhedron.
$c(n,k)$ is at least as large as the smallest subsystem of $Ax\leq b'$ that preserves the set of interior integral points. 
We have just shown that the smallest such subsystem is the entire system of inequalities, which has cardinality $|V|=\alpha(n,k)$.
Therefore, $c(n,k)\geq\alpha(n,k)$.
\end{proof}

\begin{figure}
\begin{center}
\begin{tikzpicture}[scale = 0.7]
\foreach \x in {0,...,3} {
  \foreach \y in {0,...,3} {
    \node[circle,fill=black, inner sep=0pt,minimum size=2pt] at (\x,\y) {};
  }
}

\coordinate (1) at (0,0);
\coordinate (2) at (0,1);
\coordinate (3) at (1,2);
\coordinate (4) at (3,3);
\coordinate (5) at (2,1);
\coordinate (6) at (1,0);

\draw (1)--(2)--(3)--(4)--(5)--(6)--cycle;

\begin{scope}[every node/.style={circle,fill=red,
inner sep=0pt,minimum size=5pt}]
\foreach \i in {1,...,6} {
  \node at (\i) {};
}
\end{scope}
\end{tikzpicture}
\hspace{1cm}
\begin{tikzpicture}[scale = 0.7]
\foreach \x in {0,...,3} {
  \foreach \y in {0,...,3} {
    \node[circle,fill=black, inner sep=0pt,minimum size=2pt] at (\x,\y) {};
  }
}

\coordinate (1) at (1,0);
\coordinate (2) at (2,0);
\coordinate (3) at (3,1);
\coordinate (4) at (3,2);
\coordinate (5) at (2,3);
\coordinate (6) at (1,3);
\coordinate (7) at (0,2);
\coordinate (8) at (0,1);

\draw (1)--(2)--(3)--(4)--(5)--(6)--(7)--(8)--cycle;

\begin{scope}[every node/.style={circle,fill=red,
inner sep=0pt,minimum size=5pt}]
\foreach \i in {1,...,8} {
  \node at (\i) {};
}
\end{scope}
\end{tikzpicture}
\hspace{1cm}
\begin{tikzpicture}[scale = 0.7]
\foreach \x in {0,...,4} {
  \foreach \y in {0,...,3} {
    \node[circle,fill=black, inner sep=0pt,minimum size=2pt] at (\x,\y) {};
  }
}

\coordinate (1) at (1,0);
\coordinate (2) at (0,1);
\coordinate (3) at (0,2);
\coordinate (4) at (1,3);
\coordinate (5) at (2,3);
\coordinate (6) at (4,2);
\coordinate (7) at (3,1);

\draw (1)--(2)--(3)--(4)--(5)--(6)--(7)--cycle;

\begin{scope}[every node/.style={circle,fill=red,
inner sep=0pt,minimum size=5pt}]
\foreach \i in {1,...,7} {
  \node at (\i) {};
}
\end{scope}
\end{tikzpicture}

\end{center}
\caption{This figure gives three examples of integer point configurations which imply lower bounds for $c(n,k) \geq \alpha(n,k)$.  In each example, the set $V$ is the set of integer points colored red that create the vertices of a polygon containing other integer points.  From left to right, these examples show that $\alpha(2,2) \geq 6$, $\alpha(2,4) \geq 8$, and $\alpha(2,5) \geq 7$.
}
\label{fig:alpha-examples}
\end{figure}
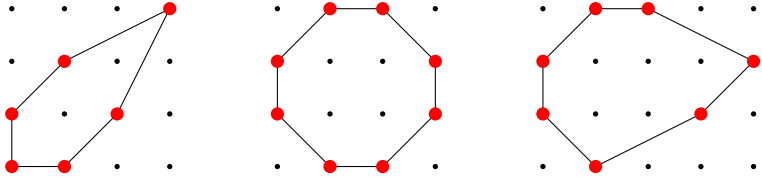

Notice that Lemma~\ref{lem:albc} is already a useful result
to quickly obtain some bounds on $c(n,k)$. For example, we
can get a much shorter proof of $c(n,1) \geq 2(2^n-1)$ than the example presented in~\cite{Aliev2016} (and
thus $c(n,1) = 2(2^n - 1)$ by our upper bound on $c(n,1)$). It suffices
to consider $V=(\{-1,0\}^n \cup \{0,1\}^n) \setminus \{0\}^n$,
and observe that $V$ is in convex position and fulfills
$|\iconv(V)\setminus V|=|\{0\}^n|=1$.
Hence, $\alpha(n,1) \geq |V| = 2(2^n-1)$, and by
Lemma~\ref{lem:albc}, $c(n,1) \geq \alpha(n,1) \geq 2(2^n-1)$.

The authors of~\cite{Aliev2016} use~\eqref{eq: qdbs bound} to prove $c(n,2)\leq 2(2^n-1)$ and leave it as an open question if this bound is tight.  We show it is indeed tight.  See~\cite[Theorem 5]{averkov2016} for similar proofs that exactly determine $c(n,k)$ for $k \leq 4$.   
\begin{proposition}\label{prop: c(n,2) lower bound}
$c(n,2) = 2(2^n-1)$.
\end{proposition}
\begin{proof}
From~\eqref{eq: qdbs bound}, or also Theorem~\ref{thm: subspaces}, we see that $c(n,2) \leq 2(2^n -1)$.  
It remains to show that $c(n,2) \geq \alpha(n,2) \geq 2(2^n-1)$.

Consider the set $V = (\{-1,0\}^n \cup \{0,1\}^n\cup\{\vec{\bf{2}}\}) \setminus \{\vec {\bf{0}}, \vec{\bf 1}\}$.  Notice that $|V| = 2(2^n -1)$, $V$ is in convex position.  We claim that $(\conv(V) \cap \Z^n) \setminus V = \{\vec {\bf{0}}, \vec{\bf 1}\}$.   To see this, let $x \in (\conv(V) \cap \Z^n)\setminus V$.  That is, $x = \sum_{v \in V_1} \lambda_v v + \sum_{u \in V_2} \mu_u u + \gamma \vec{\bf{2}}$ for some $0 \leq \lambda_v, \mu_u, \gamma < 1$ and $V_1 = \{-1,0\}^n \setminus \{\vec {\bf{0}}\}$, $V_2 = \{0,1\}^n \setminus \{\vec {\bf{0}}, \vec {\bf{1}}\}$.  Clearly, $x = \vec {\bf{0}}$ and $x = \vec{\bf 1}$ are possible.  Otherwise, note that each entry $x_i$ of $x$ satisfies $-1 \leq x_i < 2$ for $i=1, \dots, n$.  Suppose for the sake of contradiction that $x_i = -1$ for some $1 \leq i\leq n$.  Since $x_i = -1$, we must have $\mu_u = 0$ for all $u \in V_2$ and $\gamma = 0$.  But then $x \in [-1,0]^n$.  Otherwise, $0 \leq x_i \leq 1$ for all $i=1, \dots, n$, that is, $x \in [0,1]^n$.  This implies the claim.
Hence, $c(n,2)\geq \alpha(n,2) \geq 2(2^n -1)$.
\end{proof}
Upon replacing $\vec{\bf 2}$ with $\vec{\bf k}$ in the last proof one can find lower bound of $2(2^{n}-1)$ on $c(X)$ for any collinear set $X$ of $k\geq 2$ points where $\iconv(X)=X$.
Thus $c(X)=2(2^n-1)$ for any set $X$ of $k\geq 2$ collinear points with $\iconv(X)=X$, by Theorem~\ref{thm: subspaces}.

We now use the integer points in a ball to give a lower bound on $c(n,k)$.  
We begin by exhibiting a sequence that bounds from below many values of $\alpha(n,k)$, and therefore $c(n,k)$, for many values of $k$.  We will then use Lemma~\ref{lem:inductive} to connect these lower bounds to all other values $k$.  Before the main theorem, we give a lemma that collects some results about integer points in spheres.  See also~\cite{Barany2008} for a discussion of this problem and related extremal problems.

Let $B_n(u,r) = \{x \in \R^n \st \|x - u\| \leq r\}$ be the Euclidean ball centered at $u$ with radius $r$.
For a convex body $K$ and a scalar $r > 0$, let $rK:= \{r x \st x \in K\}$.  For any $u \in \R^n$, we have that $r B_n(u,1) = B_n(r u, r)$.  

\begin{lemma}
\label{lemma:ball}
Let $n \geq 1$, $u \in \R^n$ with $\|u\| < 1$.  Then for $r > 0$, the following hold:
\begin{enumerate}[(i.)]
\item \label{item:one} $|r  B_n(u,1) \cap \Z^n| = \vol(B_n(0,1)) r^n + \mathcal{E}(r)$ where $\mathcal{E}(r) = O(r^\frac{n(n-1)}{n+1})$,
\item \label{item:two}  $C_1 r^\frac{n(n-1)}{n+1} \leq  |\verts(\conv(r B_n(u,1) \cap \Z^n))|\leq C_2 r^\frac{n(n-1)}{n+1}$.
\item \label{item:three} Let $U \subseteq \R^n$ be the set of vectors $u$ such that for every $r >0$ there is at most one integer point on the boundary of the ball $r B_n(u,1)$.  The set $U$ is dense in $\R^n$.   
\item \label{item:four} Let $r > \sqrt{n}$.  Then $\conv(B_n(u,r) \cap \Z^n)\supseteq B_n(u,r - \Delta)$ for $\Delta =  \sqrt{n}$.
\item \label{item:five} Let $r >  \sqrt{n}$.  Let $e$ be an edge (1-dimensional face) of $P = \conv((B_n(u,r)\cap\Z^n)$.  Then $e$
has length at most $4\Delta^{\frac{1}{2}} R^{\frac{1}{2}}$ for $\Delta = \sqrt{n}$.  
\end{enumerate}
\end{lemma}
\begin{proof}
 (\ref{item:one}.) This follows from more general results on convex bodies that contain the origin in the interior and have positive curvature on the boundary.  This was proved in~\cite{Corput1920} for $n=2$ and~\cite[Section 8, Theorem 9]{Hlawka1950} for $n\geq 3$.  See~\cite[Section 3.1 and 3.2]{sphere_survey2006} for a survey of these results and improvements.

 (\ref{item:two}.) This follows \cite[Theorem 5]{Barany1998} since $B_n(u,1)$ is a convex body that contains the origin in the interior and has positive curvature on the boundary.

 (\ref{item:three}.) Consider any $v \in \R^n$ such that the entries $v_1, \dots, v_n$ are linearly independent over $\mathbb{Q}$.  Note that such choices of $v$ are dense in $\R^n$ by Lemma~\ref{lem:q-linear-independence}.  
 Let $L_v = \{ \lambda v \st \lambda > 0\}$.  We will prove the statement by showing that a dense subset of $L_v$ satisfies the claim.  
Let $x^1, x^2 \in \Z^n$ be distinct points and suppose that $x^1, x^2 \in \partial( r B_n(\lambda v, 1))$ for some $r,\lambda > 0$, where $\partial$ denotes the boundary.  This implies that for $i=1,2$, 
\begin{equation}
\label{eq:first}
r^2 = \| x^i - r \lambda v\|^2 = \|x^i\|^2 - 2 r \lambda x^i \cdot v + r^2 \lambda^2 \|v\|^2.
\end{equation}
Combining these equations for $i=1,2$, we see that 
$$
0 = \|x^2\|^2 - \|x^1\|^2 - 2 r \lambda (x^2 - x^1) \cdot v.
$$
Since $x^1, x^2$ are distinct and the entries of $v$ are linearly independent over $\mathbb{Q}$, we have $(x^2 - x^1) \cdot v \neq 0$.  Thus, rearranging, we have
$$
r \lambda = \frac{\|x^2\|^2 - \|x^1\|^2}{2 (x^2 - x^1) \cdot v}.
$$
Plugging this back into \eqref{eq:first}, we see that $r$, and hence also $\lambda$, are uniquely determined by $x^1$,$x^2$, and $v$.  
Therefore, any pair $x^1,x^2 \in \Z^n$, has at most a single choice $\lambda v \in L_v$ such that $x^1,x^2 \in \partial( r B_n(\lambda v, 1))$ for some $r > 0$.    Since set of pairs of lattice points, that is, $\Z^n \times \Z^n$, is a countable set and $L_v$ is an uncountable set, there is a dense subset of $L_v$.  Since the set of $v$ with entries that are linearly independent over $\mathbb{Q}$ are dense in $\R^n$, the result follows by taking the union over all such rays $L_v$.

(\ref{item:four}.)  
If this is not the case, then, since these are convex sets, there exists an inequality   $a \cdot x \leq b$ valid for $\iconv(B_n(u,r))$, but not valid for $B_n(u,r -\Delta)$.  Let $y = \argmax\{ a \cdot x \st  x \in B_n(u, r- \Delta)\}$ be the unique maximizer of the linear functional $a \cdot x$.  Next, let $z = \tfrac{y - u}{r - \Delta}$.  Then $\|z\| = 1$.  Let $\lambda = r - \tfrac{\Delta}{2}$.  It follows that $B_n(u + \lambda z,\tfrac{\Delta}{2}) \cap  \iconv(B_n(u,r)) = \emptyset$ while $B_n(u + \lambda z,\tfrac{\Delta}{2}) \subseteq B_n(u,r)$.  But since any ball of radius $\tfrac{\Delta}{2}$ must contain an integer point, we have that $B_n(u + \lambda z,\tfrac{\Delta}{2}) \cap \Z^n \neq \emptyset$, which is a contradiction.

(\ref{item:five}.)
  Since $P  \supseteq B_n(u, r - \Delta)$, no edge of $P$ can intersect the interior of $B_n(r- \Delta)$.   Clearly the Euclidean length of any edge of $P$ is at most the length of the longest chord in $B_n(u,r) \setminus \intr(B_n(u,r-\Delta))$.
By the Pythagorean Theorem, it is easy to see that the length of the longest chord in $B_n(u,r) \setminus \intr(B_n(u,r-\Delta))$ is 
$$
2 \sqrt{ r^2 - (r-\Delta)^2} = 2 \sqrt{ 2 r \Delta - \Delta^2} \leq 4 \Delta^{\frac{1}{2}} r^{\frac{1}{2}}.
$$

\end{proof}

The following result appears with a precise lower bound in~\cite[Theorem 4]{averkov2016}.
\begin{theorem}
\label{thm:lower-bound}
For each $n\geq 2$, there exists a constant $C$, depending only on $n$, such that
$$
c(n,k) 
\geq C k^\frac{n-1}{n+1}.
$$
\end{theorem}

\begin{proof}
Throughout, we will assume that the dimension $n$ is fixed.  Hence, many of the constants used in the proof will depend on $n$.

Using Lemma~\ref{lemma:ball}(\ref{item:three}.), fix any $u \in \R^n$ with $\|u\| < 1$ such that for all positive integers $N$, there exists a radius $R$ such that $|rB_n(u,R) \cap \Z^n| = N$.
Define 
$$
N_r \coloneqq | rB_n(u,1) \cap \Z^n|, \ \ \ v_r \coloneqq | \verts(\conv(rB_n(u,1) \cap \Z^n))|, \ \ k_r \coloneqq N_r - v_r.
$$
From Lemma~\ref{lemma:ball}, we have
\begin{equation}
\label{eq:Nr}
N_r =  \vol(B_n(0,1)) r^n  + \mathcal E(r), \ \ \ \text{ where } \ \ \   \mathcal{E}(r) = O(r^{\frac{n(n-1)}{n+1}})
\end{equation}
and
\begin{equation}
\label{eq:vr}
C_1\, r^{\frac{n(n-1)}{n+1}} \leq v_r \leq C_2\, r^{\frac{n(n-1)}{n+1}}.
\end{equation}
From~\eqref{eq:Nr} and~\eqref{eq:vr}, we see that there exist constants $C_3, C_4$ such that
\begin{equation}
\label{eq:kr}
C_3\,  r^n \leq  k_r \leq C_4\, r^n.
\end{equation}
It follows that 
\begin{equation*}
\alpha(n, k_r) \geq v_r  \geq C_1 r^{\frac{n(n-1)}{n+1}} \geq C_1 \left( \frac{k_r}{C_4} \right)^{\frac{n-1}{n+1}} = C_5\, k_r^{\frac{n-1}{n+1}}.
\end{equation*}
By Lemma~\ref{lem:albc}, we have that 
\begin{equation}
\label{eq:cnkr}
c(n,k_r) \geq C_5\, k_r^{\frac{n-1}{n+1}}.
\end{equation}
This shows that for all $k$ such that $k=k_r$, for some value $r$, the theorem holds.  We will now show that we can extend the lower bound to all sufficiently large values of $k$.

Fix a positive integer $k$.   Let $\{r_i\}_{i=1}^\infty$ be a sequence of $r_i > 0$ such that $N_{r_i} = i$.  This sequence exists by choice of $u$ with Lemma~\ref{lemma:ball}(\ref{item:three}.).  Therefore, $k_{r_1} = 0$ and $k_{r_i} \to \infty$ as $i \to \infty$, although this sequence is not monotonically increasing.  Thus, there must exist $r_i$, $r_{i+1}$, such that $k$ is between $k_{r_i}$ and $k_{r_{i+1}}$.  Thus, assigning $r,R$ to $r_i, r_{i+1}$ appropriately, we have that $k_r \leq k \leq k_R$ and 
\[
|N_r - N_R| = 1.
\]
For this to hold, we must have $|r - R| \leq n$.  
Therefore, it follows from~\eqref{eq:kr} that 
\begin{equation}
\label{eq:k-kr-bound}
 k \leq k_R \leq C_4 R^n \leq C_4(r+n)^n \leq C_6\, k_r
\end{equation}
for some constant $C_6$.  The last inequality follows from the binomial theorem.
By~\eqref{eq:cnkr} and~\eqref{eq:k-kr-bound} we see that
\begin{equation}
\label{eq:cnkrk}
c(n,k_r) \geq C_5 k_r^{\frac{n-1}{n+1}} \geq C_5 \left( \frac{k}{C_6} \right)^{\frac{n-1}{n+1}} =\left( \frac{C_5}{C_6^{\frac{n-1}{n+1}}}\right) k^\frac{n-1}{n+1}.
\end{equation}

Next, we derive an upper bound for the difference $|k - k_r|$.   
First notice that 
\begin{equation}
\label{eq:k-kr}
0
\leq  k-k_r 
\leq   k_R-k_r
\leq  N_R - N_r + v_r - v_R
\leq 1 + |v_r - v_R|.
\end{equation}
We will now bound the size of $|v_r - v_R|$.
We will assume that $R > r$, as the calculation is similar for $r < R$. 
\begin{claim}
There exists a constant $C_* >0$ such that $|v_r - v_R| \leq C_* (R^{\frac{n}{2}})^{\frac{n-1}{n+1}}$.
\end{claim}
 \begin{proof}
 \renewcommand{\qedsymbol}{$\Diamond$}

Let $\bar x \in (RB_n(u,1) \setminus rB_n(u,1) ) \cap \Z^n$ be the single integer point that is not in common in both balls.
Let $V_R = \verts(\conv(RB_n(u,1)\cap\Z^n))$, $V_r = \verts(\conv(rB_n(u,1)\cap\Z^n))$.  Let $P = \conv(V_R) = \conv(RB_n(u,1)\cap\Z^n)$ and $Q = \conv(V_R \setminus \{\bar x\})$.  

Let $V_{\bar x}$ denote the set of vertices $z$ of $P$ such that there exists an edge $[\bar x, z]$ of $P$.  Let $T = \conv(V_{\bar x} \cup \{\bar x\})$.  By construction, $T \supseteq P \setminus Q$.

By Lemma~\ref{lemma:ball}(\ref{item:five}), the edges of $P$ all have length at least $4 \Delta^{\frac{1}{2}} R^{\frac{1}{2}}$ where $\Delta = \sqrt{n}$.
By convexity of $T$ and the function $\| \cdot \|$, since $\|z - \bar x\| \leq 4 \Delta^{\frac{1}{2}} R^{\frac{1}{2}}$ for all $z \in V_{\bar x}$, we have that
$$
T \subseteq B_n(\bar x, 4 \Delta^{\frac{1}{2}} R^{\frac{1}{2}}).
$$
Therefore the volume $\vol(T) \leq C_7 R^{\frac{n}{2}}$.  By~\cite{andrews1963} (see also~\cite{Barany2008}), the number of vertices of any full-dimensional lattice polytope $K\subseteq \R^n$ is at most $C_8 \vol(K)^{\frac{n-1}{n+1}}$ for some constant $C_8$.  
$T$ is a full-dimensional lattice polytope, but we need to remove the vertex $\bar x$ since it is not a vertex of $Q$.
Consider instead the lattice polytope $K = \conv((T\setminus \{\bar x\} \cap \Z^n)$ and observe that $V_R \setminus V_r \subseteq \verts(K)$.  If $K$ is not full-dimensional, then $\verts(K) =  \verts(T) \setminus \{\bar x\}$.  Thus, whether or not $K$ is full-dimensional, we obtain the bound the number of vertices of $K$ as $|\verts(K)| \leq C_8 \vol(T)^{\frac{n-1}{n+1}}$.
It follows that $|v_r - V_R| = |V_r \setminus V_R| \leq C_8 (C_7 R^{\frac{n}{2}})^{\frac{n-1}{n+1}}$.  
Setting $C_* = C_8 C_7^{\frac{n-1}{n+1}}$ completes the proof of the claim.
\end{proof}

Adjusting for the $+1$ in equation~\eqref{eq:k-kr}, the claim in conjunction with~\eqref{eq:k-kr-bound} shows that there exists a constant $C_9$ such that 
\begin{equation}
\label{eq:bound-krkR-new} 
|k-k_r|  \leq C_{9} k_r^{\frac{n-1}{2(n+1)}}.
\end{equation}

Finally, by applying induction to Lemma~\ref{lem:inductive}, we see that
\begin{equation}
\label{eq:induction-bound-new}
c(n,k) \geq c(n,k_r) - |k - k_r|.
\end{equation}
Combining equations~\eqref{eq:cnkrk},~\eqref{eq:bound-krkR-new}, and~\eqref{eq:induction-bound-new}, we have that
\[
c(n,k) \geq  c(n,k_r) - |k - k_r| \geq     \left( \frac{C_5}{C_6^{\frac{n-1}{n+1}}}\right) k^\frac{n-1}{n+1}    - C_{9} k_r^{\frac{n-1}{2(n+1)}}  \geq C k^\frac{n-1}{n+1},
\]
where,  $C$ is a constant.  This completes the proof.
\end{proof}

The bound we give above on $v_r - v_R$ is likely quite loose.  This is because we use the fact that $\conv(B_n(u,R)\cap\Z^n) \supseteq B_n(u,R-\Delta)$ for $\Delta =  \sqrt{n}$ (Lemma~\ref{lemma:ball}(\ref{item:four})) and because we bound the volume of the cap by the volume of a ball with the same radius.  The value $\Delta$ can likely decrease with the size of $R$.  For comparison, $\verts(\conv(B_n(0,r)\cap\Z^n)) \subseteq B_n(0,r) \setminus B_n(0,r-\delta)$ where $\delta \leq C r^{- \frac{n-1}{n+1}}$~\cite{balog1991} .

Using similar techniques, we obtain an upper bound on $\mu_c(n,s)$.
\begin{proposition}
\label{prop:muconvex-bound}
For every fixed $n$, $\muconvex(n,s) = O(s^{(n+1)/(n-1)})$.  Therefore, $s_{n,k} = \Omega(k^{(n-1)/(n+1)})$.
\end{proposition}
\begin{proof}
Let $n$ be fixed.  
By Lemma~\ref{lemma:ball}(\ref{item:three}), there exists $u \in \R^n$ with $\|u\| < 1$ such that for all positive integers $N$, there exists $R_N>0$ such that $|R_N B_n(u,1) \cap \Z^n| = N$.  
For every $R$, let $V_R$ denote the set of vertices of the integer hull of $R B_n(u,1)$.  
Partition the set of vertices $V_R$ into parity sets $V_{R, x} := \{ v \in V_R \st  v \equiv x \pmod 2\}$ for $x \in \{0,1\}^n$.  Let $s_R = \max\{ |V_{R,x}| \st x \in \{0,1\}^n\}$ and $S_R = V_{R,x_R}$ for some $x_R \in \{0,1\}^n$ be a parity set that
corresponds to $s_R$, i.e., $s_R = |S_R|$.

By choice of $u$, $s_{R_{N+1}} \leq s_{R_N} + 1$.  Hence, for every positive integer $s$, there exists a radius $R$ such that $s = s_R$.  Since $S_R \subseteq \Z^n$ has the same parity, all midpoints of this set are integral, that is, $M(S_R) \subseteq R B_n(u,1) \cap \Z^n$.  By the pigeonhole principle, $s_R \geq \tfrac{1}{2^n} |V_R|$.
By Lemma~\ref{lemma:ball}(\ref{item:two}),  $s_R = \Omega(R^{\frac{n(n-1)}{n+1}})$ and $ |R B_n(u,1) \cap \Z^n| = O(R^n) = O(s_R^{(n+1)/(n-1)})$.  It follows that $\muconvex(n,s) = O(s^{({n+1})/({n-1})})$.  
\end{proof}

\section{Non-monotonicity of $c(n,k)$ in $k$}\label{sec: nonmonotonic}

It is easy to see that $c(n,k)$ is nondecreasing in $n$.
So, it is natural to ask whether $c(n,k)$ is nondecreasing in $k$. 
It is not.  We will rigorously show that $c(2,5) = 7$ while $c(2,4) = 8$.  Note that the values of $c(2,k)$ for $k\in \{0, \dots, 30\}$ computationally established recently~\cite{averkov2016}.

Lemmas~\ref{lem:upc25} and \ref{lem:upc24} show that $c(2,5)\leq 7$ and $c(2,4)\leq 8$.
The reverse inequalities are established by the examples in Figure~\ref{fig:alpha-examples}.
The next proposition will help with some case analysis in the proof of Lemma~\ref{lem:upc25}
\begin{proposition}\label{prop: planar triangulation}
Let $V\subseteq\R^2$.
If the points in $V$ are all collinear then $|M(V)|\geq |V|-1$.
If the points in $V$ are not all collinear then $|M(V)|\geq 2|V|-3$.
\end{proposition}
\begin{proof}
The first claim is trivial. 
If the points are not collinear, then any triangulation of $V$ has at least $2|V|-3$ edges~\cite{de2000computational}, and each edge contains a distinct midpoint.
\end{proof}

\begin{lemma}\label{lem:upc25}
$c(2,5)\leq 7$.
\end{lemma}
\begin{proof}
Let $P\subseteq \R^2$ be a non-redundant polytope with $c(2,5)$ facets and containing five integer points.
Applying Lemma~\ref{lem: one point per facet} yields a set $V\subseteq\Z^2$ of cardinality $c(2,5)$ in convex position such that $\iconv(V)\setminus V \subseteq P\cap\Z^2 $.
We must show $|V|\leq 7$.

Partition the points in $V\cup (P\cap\Z^2)$ according to their parities.  
If every parity class has three or fewer points, then there are at most $12$ points.
Hence, $|V|\leq 7$ follows from $|V| + 5\leq 12$.
If any parity class contains five or more points, then Proposition~\ref{prop: planar triangulation} implies that $|P\cap\Z^2|\geq 7$, a contradiction.
We will now show that it is possible for a parity class to have four points, but in this case the remaining parity classes have no more than eight points combined. 
This implies $c(2,5)\leq 7$.

First, if any parity class contains four collinear points, then all five of the integer points in $P$ are collinear.
That is because adjacent integer points on a line must have different parities, so that we have seven collinear points in total, five of which are in $P$.
It follows from Theorem~\ref{thm: subspaces} that $|V|\leq 6$.

Suppose that no parity class contains four collinear points, but some parity class contains four points that are not collinear.
By Proposition~\ref{prop: planar triangulation} there are at least five (integral) midpoints among them.
As all of the midpoints are integral points in $P$, there can be no more than five, so the four points of the parity class are members of $V$.
Since the four points have only five (rather than six) distinct midpoints they must form the vertices of a parallelogram.

Consider the parity classes of the five interior integer points.
Looking at Figure~\ref{fig: c(2,5) situation} and using the fact that adjacent integer points on a line must have different parities we see that the five interior integer points are split among three parity classes in groups of size $2$, $2$, and $1$.
The points in the figure are grouped as $\{a,b\},\{d,e\},\{c\}$.

Now, one parity class of $V\cup (P\cap\Z^2)$ is taken entirely by the four vertices of the parallelogram, so we are left with the five interior points in the three remaining parity classes, with two classes containing two of the points and the third class containing one point.
We will prove that no point in $V$ lies in the parity classes with $a,b$ or $d,e$.
All together, this implies that the size of the classes are at most $4,4,2,2$ hence there are $12$ total points so $|V|\leq 12-5=7$.

\begin{figure}\label{fig: c(2,5) situation}
\begin{center}
\begin{tikzpicture}[scale = 0.7]
\tikzstyle{gridpoint} = [circle,fill=black,inner sep=0pt,minimum size=2pt]
\tikzstyle{vertpoint} = [circle,fill=red,inner sep=0pt,minimum size=5pt]
\tikzstyle{intpoint} = [circle,fill=black,inner sep=0pt,minimum size=5pt]

\foreach \x in {0,...,4} {
  \foreach \y in {0,...,2} {
    \node[gridpoint] at (\x,\y) {};
  }
}
\foreach \x/\y in {0/0,2/0,2/2,4/2} {
  \node[vertpoint] at (\x,\y) {};
}
\node[intpoint,label={[label distance=-0.1cm]above right:$a$}] at (1,1) {};
\node[intpoint,label={[label distance=-0.1cm]above right:$b$}] at (3,1) {};
\node[intpoint,label={[label distance=-0.1cm]above right:$d$}] at (1,0) {};
\node[intpoint,label={[label distance=-0.1cm]above right:$e$}] at (3,2) {};
\node[intpoint,label={[label distance=-0.1cm]above right:$c$}] at (2,1) {};
\end{tikzpicture}
\end{center}
\caption{The five interior (black) and facet (red) points for the final case in the proof of Lemma~\ref{lem:upc25}.}
\end{figure}
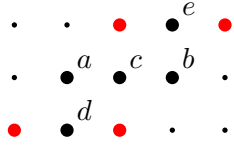
Suppose, for contradiction, that a point $z\in V$ is in the same parity class as $a$ and $b$.
We have $m_1 = \frac{1}{2}(z+a)$ and $m_2=\frac{1}{2}(z+b)$ are in $\intr(P')$.
If we show that $\{m_1,m_2\}\not\subseteq\{a,b,c,d,e\}$, then we have the desired contradiction.
We have already shown that $z$ cannot be a vertex of the parallelogram.
If $z$ is collinear with $a$ and $b$, then at least one of $\{m_1,m_2\}$ is not among $\{a,b,c,d,e\}$ which is a contradiction.
If $z$ is not collinear with $a$ and $b$, then it lies in an open half-space to one side of the line through $a$ and $b$.
The midpoints $m_1$ and $m_2$ lie in the same open half space, but none of $\{a,b,c\}$ are there and at most one of $\{d,e\}$.
Therefore, we reach the desired contradiction $\{m_1,m_2\}\not\subseteq\{a,b,c,d,e\}$.
\end{proof}

\begin{lemma}\label{lem:upc24}
$c(2,4)\leq 8$.
\end{lemma}
\begin{proof}
By Lemmas~\ref{lem:inductive} and~\ref{lem:upc25} we have
$c(2,4)\leq c(2,5)+1\leq 8$.
\end{proof}

\section*{Acknowledgement}
We would like to thank an anonymous reviewer who proposed we use Theorem~\ref{thm:schoen} and the Pl{\"u}nnecke-Ruzsa estimates to prove Proposition~\ref{prop:sum17}, reducing our previous upper bound $c(n,k)\leq O(k\log\log k / \log^{1/3}k)\cdot 2^n$ down to $c(n,k)\leq 2^nk/e^{\Omega(\log ^{1/7}k)}$.

\bibliographystyle{plain}
\providecommand\CheckAccent[1]{\accent20 #1}

{}

\end{document}